\documentclass{article}
\usepackage{amsmath,amssymb,enumerate,amsthm}

\usepackage{hyperref}

\newcommand{\R}{\mathbb{R}}

\newcommand{\N}{\mathbb{N}}
\newcommand{\ep}{\varepsilon}
\newcommand{\pa}{\partial}
\newcommand{\lr}[1]{{}\langle{}#1{}\rangle{}}

\newtheorem{theorem}{Theorem}[section]
\newtheorem{lemma}[theorem]{Lemma}
\newtheorem{proposition}[theorem]{Proposition}
\newtheorem{corollary}[theorem]{Corollary}
\theoremstyle{remark}
\newtheorem{remark}{Remark}[section]
\theoremstyle{definition}

\newtheorem{definition}{Definition}[section]

\numberwithin{equation}{section}

\makeatletter
\def\@cite#1#2{[{{\bfseries #1}\if@tempswa , #2\fi}]}
\makeatother                  

\begin{document}
{\Large\bf
\begin{center}
On global existence for semilinear wave
equations with space-dependent
critical damping
\end{center}
}

\vspace{5pt}

\begin{center}
Motohiro Sobajima%
\footnote{
Department of Mathematics, 
Faculty of Science and Technology, Tokyo University of Science,  
2641 Yamazaki, Noda-shi, Chiba, 278-8510, Japan,  
E-mail:\ {\tt msobajima1984@gmail.com}}
\end{center}

\newenvironment{summary}{\vspace{.5\baselineskip}\begin{list}{}{%
     \setlength{\baselineskip}{0.85\baselineskip}
     \setlength{\topsep}{0pt}
     \setlength{\leftmargin}{12mm}
     \setlength{\rightmargin}{12mm}
     \setlength{\listparindent}{0mm}
     \setlength{\itemindent}{\listparindent}
     \setlength{\parsep}{0pt}
     \item\relax}}{\end{list}\vspace{.5\baselineskip}}
\begin{summary}
{\footnotesize {\bf Abstract.}
The global existence for 
semilinear wave equations with 
space-dependent critical damping 
$\partial_t^2u-\Delta u+\frac{V_0}{|x|}\partial_t u=f(u)$ 
in an exterior domain is dealt with, where 
$f(u)=|u|^{p-1}u$ and $f(u)=|u|^p$ are in mind. 
Existence and non-existence of global-in-time solutions 
are discussed.  
To obtain global existence, 
a weighted energy estimate for the linear problem 
is crucial. 
The proof of such a weighted energy estimate 
contains an alternative proof of 
energy estimates established 
by Ikehata--Todorova--Yordanov [J.\ Math.\ Soc.\ Japan (2013), 183--236] but this clarifies the precise independence of the location of the support of initial data.
The blowup phenomena is verified by using a test function method 
with positive harmonic functions satisfying the Dirichlet boundary condition.
}
\end{summary}



\section{Introduction}
In this paper we consider
the following initial-boundary value problem  
\begin{align}\label{P}
\begin{cases}
\pa_t^2u(x,t)-\Delta u(x,t) +\dfrac{V_0}{|x|}\pa_tu(x,t)
=f\big(u(x,t)\big)
& \text{in}\ \Omega\times(0,T),
\\
u(x,t)=0& \text{on}\ \pa\Omega\times(0,T), 
\\
(u_0,u_1)(0)=(u_0,u_1)& \text{in}\ \Omega, 
\end{cases}
\end{align}
where $\Omega$ is an exterior domain in $\R^N$ $(N\geq 3)$ 
with smooth boundary $\pa\Omega$ and $0\notin \overline{\Omega}$,
and $f(u)=|u|^{p-1}u$ (or $f(u)=|u|^p$). 
The constant $V_0>0$ describes the effect of the damping term. 
If we only focus on the scaling structure of the equation in \eqref{P}, 
then the solution $u$ and the scale parameter $\lambda>0$ 
give another solution $\lambda^{-\frac{2}{p-1}}u(\lambda x,\lambda t)$. 
In this sense, the damping term of the form 
$|x|^{-1}$ is scale-critical 
and therefore the constant in front of $|x|^{-1}$ plays a crucial role. 
Despite of this, the damping coefficient 
$V_0|x|^{-1}$ is bounded because of the setting for the domain $\Omega$.
Our interest is that 
the problem \eqref{P} having the scale-critical damping term
 possesses nontrivial 
global-in-time solutions or not. 

If $\Omega=\R^N$, $V_0=0$, then the problem \eqref{P} is the
semilinear wave equation
with power type nonlinearities
\begin{align}\label{eq:usualwave}
\begin{cases}
\pa_t^2u(x,t)-\Delta u(x,t) 
=|u(x,t)|^p
& \text{in}\ \R^N\times(0,T),
\\
(u_0,u_1)(0)=(u_0,u_1)& \text{in}\ \R^N, 
\end{cases}
\end{align}
It is proved in John \cite{John1979} for the three dimensional case that 
\begin{itemize}
\item if $1<p<1+\sqrt{2}$, then \eqref{eq:usualwave} does not have nontrivial 
global-in-time solutions;
\item if $p>1+\sqrt{2}$, then \eqref{eq:usualwave} possesses 
nontrivial global-in-time solutions. 
\end{itemize}
There are many subsequent papers 
dealing with the semilinear wave equations 
in $N$-dimensional space
(see e.g., 
Kato \cite{Kato1980}, Yordanov--Zhang \cite{YZ2006}, Zhou \cite{Zhou2007}
and their references therein), and then 
the critical exponent, that is, 
the threshold for dividing existence and non-existence of global-in-time solutions is clarified as 
the positive root of the quadratic equation
\[
\gamma(N,p)=2+(N+1)p-(N-1)p^2=0
\]
which is so-called  Strauss exponent given by 
\[
p_S(N)=\frac{N+1+\sqrt{N^2+10N-7}}{2(N-1)}\quad (N\geq2).
\]

In the case of the semilinear wave equation with constant damping 
\begin{align}\label{eq:usualdw}
\begin{cases}
\pa_t^2u(x,t)-\Delta u(x,t) 
+\pa_tu(x,t)=|u(x,t)|^{p}
& \text{in}\ \Omega\times(0,T),
\\
u(x,t)=0& \text{in}\ \pa\Omega\times(0,T),
\\
(u_0,u_1)(0)=(u_0,u_1)& \text{in}\ \Omega, 
\end{cases}
\end{align}
in Todorova--Yordanov \cite{ToYo2001}, 
Zhang \cite{Zhang2001} and Ikehata--Tanizawa \cite{IT2005},
it is proved for \eqref{eq:usualdw} with $\Omega=\R^N$ that
existence of global-in-time solutions if $1+\frac{2}{N}<p<\frac{N+2}{N-2}$
and non-existence of those if $1<p\leq 1+\frac{2}{N}$. 
This threshold $1+\frac{2}{N}$ is exactly the same as 
the Fujita exponent for the semilinear heat equation found 
in Fujita \cite{Fujita1966}. 
A certain low regularity solutions of \eqref{eq:usualdw} 
in the class $(H^{\alpha,0}\cap H^{0,\delta})\times(H^{\alpha-1,0}\cap H^{0,\delta})$
is considered in Hayashi--Kaikina--Naumkin \cite{HKN2004}, 
where $\mathcal{F}$ is the Fourier transform and
\[
H^{\ell,m}=\{\phi\in L^2(\R^N)\;;\;\lr{x}^{m}\lr{|\nabla|}^{\ell}\phi\in L^2\}, 
\quad \lr{|\nabla|}=\mathcal{F}^{-1}\lr{\xi}\mathcal{F}, 
\quad \lr{\xi}=\sqrt{1+|\xi|^2}.
\]
The exterior domain case has been considered 
in Ono \cite{Ono2003} 
by using the result of Dan--Shibata \cite{DS1995}. 
The existence result for the exterior problem \eqref{eq:usualdw}
with $N=2$, $2<p<\infty$ is given in Ikehata \cite{Ikehata2004JDE}. 
The non-existence result for this problem
with $N\geq 2$ and $1<p<1+\frac{2}{N}$ 
is proved in Takeda--Ogawa \cite{OT2009}. Therefore 
the critical exponent for \eqref{eq:usualdw} is 
given by the Fujita exponent $p=1+\frac{2}{N}$.

We move to the problem with space-dependent damping. 
The problem is the following:
\begin{align}\label{eq:usualdw2}
\begin{cases}
\pa_t^2u(x,t)-\Delta u(x,t) 
+\dfrac{a}{\lr{x}^{\alpha}}\pa_tu(x,t)=|u(x,t)|^{p}
& \text{in}\ \R^N\times(0,T),
\\
(u_0,u_1)(0)=(u_0,u_1)& \text{in}\ \R^N, 
\end{cases}
\end{align}
where $a>0$ is positive constant. In this case, Ikehata--Todorova--Yordanov \cite{ITY2009} 
proved that $p=1+\frac{2}{N-\alpha}$ is the critical exponent for this problem for $\alpha\in (0,1)$. This result can be regarded as a generalization 
of constant damping case $\alpha=0$. 
The critical exponent for the case $\alpha<0$ is still given by $p=1+\frac{2}{N-\alpha}$ in Nishihara--Sobajima--Wakasugi \cite{NSW2018}. 
On the other hand, 
if $\alpha>1$, 
then 
Li--Tu \cite{LT2020} proved 
that the critical exponent for \eqref{eq:usualdw2} is given by $p=p_S(N)$.
This means that the case $\alpha>1$ is close to the problem \eqref{eq:usualwave}. Therefore $\alpha=1$ is the threshold 
for the structure of the critical exponent. 
If $\alpha=1$, then 
Ikehata--Todorova--Yordanov \cite{ITY2013}, 
proved that 
the linear solutions with compactly supported initial data 
satisfies the following energy estimate
\begin{align*}
\int_{\R^N}
\Big(|\nabla u(x,t)|^2+(\pa_tu(x,t))^2\Big)
\,dx\leq 
\begin{cases}
C(1+t)^{-a}, &\text{if}\ 1<a<N, 
\\
C_\ep(1+t)^{-N+\ep}, &\text{if}\ a\geq N,
\end{cases}
\end{align*}
where the constants $C$ and $C_\ep$ ($\ep$ is arbitrary) 
depend on the location of the support of initial data. 
One can observe that the behavior of solutions strongly depends on
the size of the constant $V_0$. 
For the semilinear problem, non-existence of global-in-time solutions 
for the case $1<p\leq 1+\frac{2}{N-1}$ can be found in Li \cite{Li2013}. 

Then  we come back to our problem \eqref{P}. 
In the whole space case $\Omega=\R^N$ ($N\geq 3$), in Ikeda--Sobajima \cite{IkSo_FE}, 
non-existence of global-in-time solutions to \eqref{P} when 
$0<V_0<\frac{(N-1)^2}{N+1}$ and $\frac{N}{N-1}<p\leq p_S(N+V_0)$. 
It should be notice that 
although the critical exponent is not determined yet, 
one can find that the structure of solutions to \eqref{P} strongly depends on 
the parameter $V_0$. 
A further development about problems with 
singular dampings and potentials can be found in Dai--Kubo--Sobajima \cite{DKS2021}. 

As a summary, existence of global-in-time solutions 
to \eqref{P} is still open so far. 
To discuss such a situation, we need to clarify 
the dependence of 
the behavior of linear solutions 
with respect to the parameter $V_0$. 
The purpose of the present paper is to address 
existence of global-in-time solutions to \eqref{P}. 
Although we can treat the singular damping case $\Omega=\R^N$, 
to avoid a complicated discussion, 
we only deal with the case of exterior domain $\Omega$ with 
smooth boundary $\pa\Omega$ and $0\notin \overline{\Omega}$.  

Before stating our result, 
we give a definition of solutions to \eqref{P}. 
\begin{definition}
The function $u$ is called a weak solution 
of the problem \eqref{P} in $(0,T)$ if 
\[
u\in C([0,T);H_{0}^1(\Omega))\cap 
C^1([0,T);L^2(\Omega))
\cap C([0,T);L^{2p}(\Omega))
\]
and the pair $\mathcal{U}(t)={}^t(u(t),\pa_tu(t))$ satisfies
\[
\mathcal{U}(t)
=
e^{-t\mathcal{A}}
\begin{pmatrix}
u_0\\ 
u_1
\end{pmatrix}
+
\int_{0}^t
e^{-(t-s)\mathcal{A}}[\mathcal{N}(\mathcal{U}(s))]
\,ds, 
\]
where $\{e^{-t\mathcal{A}}\}_{t\geq 0}$ 
is the $C_0$-semigroup on the Hilbert space 
$\mathcal{H}=H^1_0(\Omega)\times L^2(\Omega)$ 
generated by 
\[
-\mathcal{A}=\begin{pmatrix}
0&1\\ 
\Delta & -V_0/|x|
\end{pmatrix}
\]
endowed with the domain 
$D(\mathcal{A})=[H^2(\Omega)\cap H^1_0(\Omega)]\times H^1_0(\Omega)$
and $\mathcal{N}({}^t(u,v))={}^t(0,|u|^{p-1}u)$. 
If $T<\infty$, then $u$ is called a local-in-time weak solution, 
and 
if $T=\infty$, then $u$ is called a global-in-time weak solution. 
\end{definition}
By the standard argument based on the semigroup theory 
(see e.g., Cazenave--Haraux \cite{CHbook}, or Ikeda--Sobajima \cite{IkSo_FE}), 
it is not difficult to find local-in-time (weak) solutions 
of \eqref{P} for every initial data. 
Moreover, it is well-known that 
the problem \eqref{P} has the finite propagation property. 
\begin{proposition}\label{prop:localsol}
Assume that $V_0\geq 0$ and $1<p<\frac{N}{N-2}$. 
Then for every ${}^t(u_0,u_1)\in H_0^1(\Omega)\times  L^2(\Omega)$, 
there exists a positive constant $T=T(u_0,u_1)\in (0,\infty]$ such that 
\eqref{P} has a unique weak solution $u$ in $(0,T)$. 
Moreover, 
if ${\rm supp}\,u_0\cup {\rm supp}\,u_1\subset \overline{B(0,R_0)}$, 
then for every $t\in (0,T)$, ${\rm supp}\,u(t)\subset \overline{B(0,R_0+t)}$. 
\end{proposition}

Now we are in a position to state our result 
about existence of global-in-time weak solutions 
to \eqref{P}. 
\begin{theorem}\label{thm:main}
Assume that $V_0> N-2$ and $1+\frac{4}{N-2+\min\{N,V_0\}}< p<\frac{N}{N-2}$. 
Then there exists a positive constant $\delta_0$ such that 
if the pair $(u_0,u_1)\in H^1_0(\Omega)\times L^2(\Omega)$ satisfies 
\begin{equation}\label{ass:u_0u_1}
\int_{\Omega}\Big(|\nabla u_0|^2+u_1^2\Big)(1+|x|)^{\frac{4}{p-1}-N+2}\,dx
\leq \delta
\end{equation}
for some $\delta\in (0,\delta_0]$, 
then the problem \eqref{P} possesses 
a unique global-in-time (weak) solution $u$.
Moreover, $u$ satisfies 
\begin{align}\label{est:main}
\int_{\Omega}\Big(|\nabla u(t)|^2+(\pa_tu(t))^2\Big)(1+t+|x|)^{\frac{4}{p-1}-N+2}\,dx
\leq C\delta,
\end{align}
where $C$ is a positive constant depending only on $N$, $p$ and $V_0$. 
\end{theorem}
\begin{remark}
The same assertion as Theorem \ref{thm:main} for \eqref{eq:usualdw2} 
with $\alpha=1$ can be verified by the almost same procedure. 
\end{remark}

\begin{remark}
The rate for the energy decay 
of global-in-time solutions to \eqref{P} is (almost) the same as 
the linear estimates proved by
Ikehata--Todorova--Yordanov \cite{ITY2013}. 
This suggests that the solution of the semilinear problem 
with small initial data behaves like the one of the linear problem. 
\end{remark}
\begin{remark}
The weight $|x|^{\frac{4}{p-1}-N+2}$ may be reasonable in view of scaling
structure for the problem \eqref{P}.  
\end{remark}

If we focus our attention to the case $p>1+\frac{2}{N-1}$, 
we can deduce the following corollary from Theorem \ref{thm:main}.  
\begin{corollary}
Assume that $V_0\geq N$, $\frac{N+1}{N-1}<p<\frac{N}{N-2}$.
Then there exist positive constants $\delta_0$ and $C$ such that 
if the pair $(u_0,u_1)\in H^1_0(\Omega)\times L^2(\Omega)$ satisfies 
\eqref{ass:u_0u_1} for some $\delta\in (0,\delta_0]$, 
then the problem \eqref{P} possesses 
a unique global-in-time (weak) solution $u$ satisfying \eqref{est:main}. 
\end{corollary}

The following is a part of 
non-existence result for \eqref{P} with subcritical nonlinearity $1<p\leq 1+\frac{2}{N-1}$. 
This clarifies that the critical exponent for the problem \eqref{P}
with $V_0\geq N$ is $p=1+\frac{2}{N-1}$.
\begin{proposition}\label{prop:blowup}
Assume that $N\geq 3$, $V_0>0$ and $1<p\leq 1+\frac{2}{N-1}$.
If $\Omega=\R^N\setminus \overline{B(0,1)}$, then the following assertion holds:
for every ${}^t(u_0,u_1)\in [C_0^\infty(\Omega)]^2$ satisfying 
\[
\int_{\Omega}\Big(u_0+\frac{V_0}{|x|}u_1\Big)(1-|x|^{2-N})\,dx>0,
\]
the corresponding solution $u$ of \eqref{P} blows up at finite time.
\end{proposition}

Here we briefly describe the (very rough) idea of the proof. 
At first, we consider the estimates of weighted energy functional 
\[
\int_{\Omega}\Big(|\nabla w|^2+(\pa_tw)^2\Big)(1+t+|x|)^{m}\,dx
\]
($0<m<N$) for the following inhomogeneous problem 
\begin{align}\label{eq:inhomo}
\begin{cases}
\pa_t^2w-\Delta w+\dfrac{V_0}{|x|}\pa_tw=F &\text{in}\ \Omega\times (0,T), 
\\
w=0 &\text{on}\ \pa\Omega\times (0,T), 
\\
(w,\pa_tw)(0)=(w_0,w_1),
\end{cases}
\end{align}
where $w_0,w_1,F$ are compactly supported. If $V_0< N-1$, then 
we can prove the desired estimate (Proposition \ref{prop:3-1}) 
by using the similar idea by Ikehata--Todorova--Yordanov \cite{ITY2013}. 
To treat the case $V_0>N-1$, we introduce the decomposition of the solution of  \eqref{eq:inhomo} as the following
\[
\begin{cases}
\dfrac{\lambda}{|x|^2}\psi_1-\Delta\psi_1=w_1&\text{in}\ \Omega, 
\\
\psi_1=0&\text{on}\ \pa\Omega,
\end{cases}
\quad 
\begin{cases}
\dfrac{V_0}{|x|}\pa_tv-\Delta v=F&\text{in}\ \Omega\times (0,T), 
\\
v=0&\text{on}\ \pa\Omega,
\\
v(0)=w_0+\frac{\lambda}{V_0|x|}\psi_1
\end{cases}
\]
and 
\begin{align*}\label{eq:inhomo2}
\begin{cases}
\pa_t^2U-\Delta U+\dfrac{V_0}{|x|}\pa_tU=-\pa_tv &\text{in}\ \Omega\times (0,T), 
\\
U=0 &\text{on}\ \pa\Omega\times (0,T), 
\\
(U,\pa_tU)(0)=(-\psi_1,-\frac{\lambda}{V_0|x|}\psi_1).
\end{cases}
\end{align*}
Then we can find the relation $w=v+\pa_tU$. The idea of such a decomposition 
is based on Sobajima \cite{Soba_AE} and Ikehata--Sobajima \cite{IkeSob_SLP}
(motivated by so-called modified Morawetz method in Ikehata--Matsuyama \cite{IM2002}). 
The weighted estimates for $v$ (and $\pa_tv$) are valid via 
a weighted energy estimate due to 
Sobajima--Wakasugi \cite{SoWa_CCM}. Then combining these estimates 
and energy estimates for $0<V_0<N-1$, we can deduce the desired energy estimates for the case $V_0>N-1$. 
To justify the above procedure, 
we need to use the restriction on the bounded region $D=\{x\in \Omega\;;\;|x|<R_0+T\}$ for $t\in (0,T)$ via the finite propagation property.

To apply the above estimates to the semilinear problem \eqref{P}, we use 
the Caffarelli--Kohn--Nirenberg inequality of the form 
\[
C\int_{\Omega}|u|^{2p}|x|^{\mu'}\,dx
\leq 
C\int_{\Omega}|\nabla u|^{2}|x|^{\mu}\,dx.
\]
A priori estimate for the weighted energy 
with the blowup alternative provides the global existence. 
The proof of small data blowup \ \eqref{prop:blowup} 
is an application to the test function method 
with the positive harmonic function satisfying Dirichlet boundary condition, 
which is used in Ikeda--Sobajima \cite{IS_NA}.  
To justify the above procedure, 
we introduce the problem in bounded domain 
in view of finite propagation property.

The paper is organized as follows. 
In Section \ref{sec:prelim}, 
we collect some functional inequalities 
(weighted Hardy inequalities and Caffarelli--Kohn--Nirenberg inequalities)
and 
a family of special solutions to the parabolic equation $\frac{V_0}{|x|}\pa_t\Phi=\Delta \Phi$. 
Section \ref{sec:inhomo} is devoted to 
the proof of weighted energy estimates 
for the inhomogeneous problem \eqref{eq:inhomo}. 
In Section \ref{sec:GE}, we prove existence of global--in--time solutions 
to \eqref{P} via a priori estimate with the blowup alternative.
The small data blowup (non-existence) of solutions to \eqref{P}
is discussed in Section \ref{sec:SDBU}.

\section{Preliminaries in bounded domains}\label{sec:prelim}

In this section, we collect some important lemmas 
to analyse the problem \eqref{P}. 
\subsection{Functional inequalities}

Here we give some functional inequalities on
bounded domain $D\subset \R^N$ such that 
$\pa D$ is smooth enough and $0\notin \overline{D}$ . 

The first inequality is so-called Hardy inequality.
For the proof, we refer Metafune--Sobajima--Spina \cite[Proposition 8.1]{MSS2015}. 
\begin{lemma}\label{lem:hardy}
If $N-2+\beta>0$, then 
for every $w\in H^1_0(D)$, 
\[
\left(\frac{N-2+\beta}{2}\right)^2\int_{D}w^2|x|^{\beta-2}\,dx
\leq \int_{D}|\nabla w|^2|x|^{\beta}\,dx.
\]
\end{lemma}
The second is the Gagliardo--Nirenberg inequality (see e.g., Brezis \cite[Theorem 9.9]{Brezisbook}).
\begin{lemma}\label{lem:GN}
If $N\geq 3$, then there exists a positive constant 
$C_{GN}$ depending only on $N$ such that 
for every $H^1_0(D)$, 
\[
\left(
\int_{D}|w|^{\frac{2N}{N-2}}
\,dx\right)^{\frac{N-2}{2N}}
\leq C_{GN}
\left(
\int_{D}|\nabla w|^{2}
\,dx\right)^{\frac{1}{2}}
\]
\end{lemma}
The third is a part of 
the Caffarelli--Kohn--Nirenberg inequality 
(see Caffarelli--Kohn--Nirenberg \cite{CKN1984}), 
which is crucial to treat the nonlinear effect from $|u|^{p-1}u$ in \eqref{P}. 
For the reader's convenience, we will give a proof
of that via the Hardy and Gagliardo--Nirenberg inequalities. 
\begin{lemma}\label{lem:embed-1}
If $\mu>2-N$ and $q\in (2,\frac{2N}{N-2})$, then 
there exists a positive constant $C_{GN,\mu}$ 
(depending only on $N,\mu$) such that 
for every $w\in H^1_0(D)$, 
\begin{align}\label{eq:GN-mu}
\int_{D}|u|^{q}|x|^{\mu'}\,dx
\leq 
C_{GN,\mu}
\left(
\int_{D}|\nabla u|^2|x|^{\mu}\,dx
\right)^\frac{q}{2}, 
\end{align}
where 
$\mu'=\mu-2+\frac{N-2+\mu}{2}(q-2)$.
\end{lemma}
\begin{proof}
Let $\gamma\in \R$ be given later. 
We see from H\"older inequality that 
\[
\int_{\Omega}|u|^{q}|x|^{\gamma}\,dx
\leq 
\left(
\int_D |u|^{2}|x|^{\gamma'-2}\,dx
\right)^{1-\frac{N-2}{4}(q-2)}
\left(
\int_D \big||x|^{\frac{\gamma'}{2}}u\big|^{\frac{2N}{N-2}}\,dx
\right)^{\frac{N-2}{4}(q-2)},
\]
where $\gamma'=\gamma+2-\frac{N-2+\gamma}{2}(q-2)$.
Now we choose $\gamma'=\mu$, that is, $\gamma=\mu'$. Then 
using the Hardy inequality (Lemma \ref{lem:hardy})
and the Gagliardo--Nirenberg inequality (Lemma \ref{lem:GN}), 
we obtain \eqref{eq:GN-mu}. The proof is complete. 
\end{proof}

\subsection{Special solutions to the corresponding parabolic equation}

Next we introduce
a family of special solutions of homogeneous parabolic equation 
\[
\frac{V_0}{|x|}\pa_t\Phi-\Delta\Phi=0
\quad \text{in}\ \R^N\times [0,\infty), 
\]
which is used in Sobajima--Wakasugi \cite{SoWa_CCM}.  

\begin{definition}\label{Phi}
Define a family of functions $\{\Phi_\beta\}_{\beta\in\R}$ as follows:
\[
\Phi_{\beta}(x,t)=(1+t)^{-\beta}\varphi_\beta\left(\frac{|x|}{V_0(1+t)}\right), 
\quad 
\varphi_\beta(z)=e^{-z}M\left(N-1-\beta,N-1;z\right),
\]
where $M(a,c;z)$ denotes the Kummer confluent hypergeometric function
\[
M(a,c;z)=\sum_{n=0}^\infty\frac{(a)_n}{(c)_n}\frac{z^n}{n!}
\quad 
\text{for}\ a\in \R,\ -c\notin\N\cup\{0\}
\]
and $(d)_n$ is the Pochhammer symbol
given by $(d)_0=1$ and $(d)_n=\prod_{k=1}^n(d+k-1)$ for $n\in\N$
(for the detail, see e.g., Beals--Wong \cite{BWbook}). 
\end{definition}
The properties of the functions $\Phi_\beta$ is listed as follows. 
\begin{lemma}[{\cite[Lemma 2.4]{SoWa_CCM}}]\label{lem:Phi}
The following assertions hold:
\begin{itemize}
\item[\bf (i)] For every $\beta\in \R$, 
\[
\frac{V_0}{|x|}\pa_t\Phi_\beta-\Delta\Phi_\beta=0, \quad x\in \R^N, \ t\geq0.
\]
\item[\bf (ii)] For every $\beta\in \R$, 
\[
\pa_t\Phi_\beta=-\beta \Phi_{\beta+1}, \quad x\in \R^N, \ t\geq0.
\]
\item[\bf (iii)] For every $\beta\in \R$, 
\[
|\Phi_\beta|\leq C_{\Phi_\beta} \Big(1+t+V_0|x|\Big)^{-\beta}, \quad x\in \R^N, \ t\geq0.
\]
\item[\bf (iv)] For every $\beta<N-1$, 
\[
\Phi_\beta\geq c_{\Phi_\beta} \Big(1+t+V_0|x|\Big)^{-\beta}, \quad x\in \R^N, \ t\geq0.
\]
\end{itemize}
\end{lemma}

Moreover, we need the following two lemmas  
which comes from integration by parts.
\begin{lemma}[{\cite{Soba_DIE}}]\label{lem:ibp}
Assume that $\Phi\in C^2(\overline{D})$ is positive and
$\delta\in (0,\frac{1}{2})$. 
Then for every $z\in H^2(D)\cap H^1_0(D)$, 
\begin{align*}
\int_D \frac{z\Delta z}{\Phi^{1-2\delta}}\,dx
\leq 
\frac{\delta}{1-\delta}
\int_D \frac{|\nabla z|^2}{\Phi^{1-2\delta}}\,dx
+
\frac{1-2\delta}{2}
\int_D \frac{z^2\Delta \Phi}{\Phi^{2-2\delta}}\,dx.
\end{align*} 
\end{lemma}
\begin{lemma}[{\cite[Lemma 3.5]{SoWa_CCM}}]\label{lem:hardy2}
If $m>2-N$, then 
for every $z\in H_0^1(D)$ and $t\geq0$, 
\begin{align*}
\int_D z^2\frac{\Psi(t)^{m-1}}{|x|}\,dx
\leq 
\min\left\{\frac{N-1}{2},\frac{N-2+m}{2}\right\}^{-2}\int_D |\nabla z|^2\Psi(t)^{m}\,dx,
\end{align*} 
where $\Psi=1+t+|x|$. 
\end{lemma}

\section{The inhomogeneous problems in bounded domains}\label{sec:inhomo}
In this section, 
we consider the following inhomogeneous 
problem in the bounded domain $D$ 
($\pa D$ is smooth enough and $0\notin \overline{D}$):
\begin{align}\label{AP}
\begin{cases}
\pa_t^2w(x,t)-\Delta w(x,t) +\dfrac{V_0}{|x|}\pa_tw(x,t)
=F(x,t)
\quad \text{in}\ D\times(0,T),
\\
w(x,t)=0\quad \text{on}\ \pa D\times(0,T), 
\\
(w,\pa_tw)(0)=(w_0,w_1),
\end{cases}
\end{align}
where $(w_0,w_1)\in H^1_0(D)\times L^2(D)$ 
and  $F\in C([0,T);L^2(D))$. 
If $F\equiv 0$ and $D=\R^N$, then 
the energy estimate of $w$ with 
compactly supported initial data is proved
in Ikehata--Todorova--Yordanov \cite{ITY2013}.
In contrast, we will show the estimate of
the functional with space-time dependent weight
\[
E_{m}^{\Psi}(w;t)
=
\int_D\Big(|\nabla w(x,t)|^2+\big(\pa_tw(x,t)\big)\Big)\Psi(x,t)^{m}\,dx, 
\quad \Psi(x,t)=1+t+|x|.
\]
To apply such estimates in the bounded domain $D$ 
to the case of exterior domain, 
it is crucial to derive them 
having constants which are independent of the shape of $D$. 

\begin{proposition}\label{prop:3-1}
Assume that 
$(w_0,w_1)\in H^1_0(D)\times L^2(D)$ 
and  $F\in C([0,T];L^2(D))$. 
If $0<m<\min\{N-1,V_0-1\}$, then 
there exist positive constants 
$\delta_{m}$,  $K_{m}$ (depending only on $N$ and $m$)
such that 
\begin{align*}
&E_{m+1}^{\Psi}(w;t)
+
\delta_m
\int_0^tE_{m}^{\Psi}(w;s)
\,ds
\\
&\leq 
K_m
\left(
E_{m+1}^{\Psi}(w;0)
+
\int_0^t\int_D F(x,s)^2\Psi(x,s)^{m+1}|x|\,dx\,ds
\right).
\end{align*}
\end{proposition}
By using the resolvent $J_n=(1-\frac{1}{n}\Delta)^{-1}$ 
of the operator $-\Delta$ with the domain $H^2(D)\cap H^1_0(D)$, 
we can verify that 
$(J_n w_0, J_n w_1)\in (H^2(D)\cap H^1_0(D))\times H^1_0(D)$ 
and $J_n F\in C([0,T];H^1_0(D))$ and 
\[
\begin{cases}
J_n w_0\to w_0 &\text{in}\ H^1_0(\Omega),
\\
J_n w_1\to w_1 &\text{in}\ L^2(\Omega),
\\
J_n F\to F &\text{in}\ C([0,T];L^2(\Omega)).
\end{cases}
\]
Therefore in the energy method, (by the above approximation) 
we can use the integration by parts. 

\subsection{Energy estimates for \texorpdfstring{$0<m<N-2$}{}}
We will give estimates of the weighted energy functionals
\begin{align*}
E_{\mu+1}(w;t)=
\int_{D}\Big(|\nabla w(x,t)|^2+\big(\pa_tw(x,t)\big)^2\Big)|x|^{\mu+1}\,dx.
\end{align*}
The following (auxiliary) functionals play an essential role:
\begin{align*}
\widetilde{E}_{\mu+1}(w;t)
&=
E_{\mu+1}(w;t)
-\frac{(m+1)(N-2+\mu)}{2}
\int_D w(x,t)^2|x|^{\mu-1}\,dx,
\\
E_{\mu}^*(w;t)
&=
\int_D
\Big(2w(x,t)\pa_tw(x,t)+\frac{V_0}{|x|}w(x,t)^2\Big)|x|^{\mu}\,dx
\end{align*}
and their linear combination
\[
E^{\sharp}_{\mu+1}(w;t)=\widetilde{E}_{\mu+1}(w;t)
+\frac{V_0}{2}E_{\mu}^*(w;t).
\]
The following lemma asserts that $E_{\mu+1}$ and $E^{\sharp}_{\mu+1}$ 
are equivalent under a suitable restriction on $\mu$. 
\begin{lemma}\label{lem:energy0}
If $3-N<\mu<\sqrt{(N-2)^2+1+V_0^2}$, then 
there exist positive constants $c_{\mu+1}^\sharp$ and $C_{\mu+1}^\sharp$ 
depending only on $N, \mu$ and $V_0$ such that 
\begin{align*}
c_{\mu+1}^\sharp E_{\mu+1}(w;t)
\leq 
E_{\mu+1}^{\sharp}(w;t)
\leq 
C_{\mu+1}^\sharp E_{\mu+1}(w;t).
\end{align*}
\end{lemma}
\begin{proof}
Since for every $\ep>0$, the Young inequality gives
\[
V_0
\left|\int_D w\pa_tw|x|^{\mu}\,dx\right|
\leq 
\frac{1}{1+\ep}
\int_D(\pa_tw)^2|x|^{\mu+1}\,dx
+
\frac{V_0^2}{4}(1+\ep)
\int_D w^2|x|^{\mu-1}\,dx, 
\]
we have
\begin{align*}
E^{\sharp}_{\mu+1}(w;t)
&=
\int_D |\nabla w|^2|x|^{\mu+1}\,dx
-\frac{(\mu+1)(N-2+\mu)}{2}
\int_D w^2|x|^{\mu-1}\,dx
\\
&\quad
+
\int_D(\pa_tw)^2|x|^{\mu+1}\,dx
+V_0\int_D w\pa_tw|x|^\mu\,dx 
+ \frac{V_0^2}{2}
\int_D w^2|x|^{\mu-1}\,dx
\\
&\geq 
\int_D |\nabla w|^2|x|^{\mu+1}\,dx
+
\frac{\ep}{1+\ep}\int_D (\pa_tw)^2|x|^{\mu+1}\,dx
\\
&\quad
+\left(\frac{V_0^2}{4}(1-\ep)-\frac{(\mu+1)(N-2+\mu)}{2}\right)
\int_D w^2|x|^{\mu-1}\,dx.
\end{align*}
Moreover, using Lemma \ref{lem:hardy} with $\beta=\mu+1>2-N$, we see
\begin{align*}
E^{\sharp}_{\mu+1}(w;t)
&\geq 
\ep \int_D |\nabla w|^2|x|^{\mu+1}\,dx
+
\frac{\ep}{1+\ep}\int_D (\pa_tw)^2|x|^{\mu+1}\,dx
\\
&\quad 
+
\frac{1}{4}
\Big[(N-2)^2+1+V_0^2-\mu^2
-\big((N-1+\mu)^2+V_0^2)\ep
\Big]
\int_D w^2|x|^{\mu-1}\,dx.
\end{align*}
Choosing $\ep>0$ small enough, we deduce the desired lower bound.
The calculation for the upper bound is almost the same as above. 
\end{proof}
To provide estimates for the derivatives of $\widetilde{E}_{\mu+1}$ and $E_\mu^*$, 
we need to introduce a weighted gradient
\[
\nabla_{\!\mu}w=\nabla w+\frac{N-2+\mu}{2}\frac{x}{|x|^2}w
\]
(if we consider the problem \eqref{eq:usualdw2} with $\alpha=1$, 
then we choose $\nabla_{\!\mu}w=\nabla w+\frac{N-2+\mu}{2}\frac{x}{\lr{x}^2}w$).
The following lemma describes the relation between 
$\nabla$ and $\nabla_{\!\mu}$ in the sense of weighted $L^2$-norms. 
\begin{lemma}\label{lem:D_m}
If $\mu>2-N$, then for every $w\in H_0^1(D)$, 
\begin{gather*}
\frac{1}{2}
\int_D|\nabla w|^2|x|^{\mu}\,dx
\leq 
\int_D|\nabla_{\!\mu}w|^2|x|^{\mu}\,dx
+
\left(\frac{N-2+\mu}{2}\right)^2\int_D w^2{|x|^{\mu-2}}\,dx,
\\
5\int_D|\nabla w|^2|x|^{\mu}\,dx\geq 
\int_D|\nabla_{\!\mu}w|^2|x|^{\mu}\,dx
+
\left(\frac{N-2+\mu}{2}\right)^2\int_D w^2{|x|^{\mu-2}}\,dx.
\end{gather*}
\end{lemma}
\begin{proof}
Using the Young inequality, we have 
\begin{align*}
\int_D|\nabla w|^2|x|^{\mu}\,dx
&=
\int_D\left|\nabla_{\!\mu} w-\frac{N-2+\mu}{2}\frac{x}{|x|^2}w\right|^2|x|^{\mu}\,dx
\\
&=
\int_D\left(|\nabla_{\!\mu} w|^2-(N-2+\mu)\nabla_{\!\mu}w\cdot\frac{x}{|x|^2}w+\left(\frac{N-2+\mu}{2}\right)^{2}\frac{w^2}{|x|^2}\right)|x|^{\mu}\,dx
\\
&\leq 
2\int_D\left(|\nabla_{\!\mu} w|^2+\left(\frac{N-2+\mu}{2}\right)^{2}\frac{w^2}{|x|^2}\right)|x|^{\mu}\,dx.
\end{align*}
As in the same way, we see that 
\[
\int_D|\nabla_{\!\mu} w|^2|x|^{\mu}\,dx
\leq 
2\int_D\left(|\nabla w|^2+\left(\frac{N-2+\mu}{2}\right)^{2}\frac{w^2}{|x|^2}\right)|x|^{\mu}\,dx.
\]
Using Lemma \ref{lem:hardy} with $\beta=\mu$, we obtain 
\[
\int_D|\nabla_{\!\mu} w|^2|x|^{\mu}\,dx+
\left(\frac{N-2+\mu}{2}\right)^{2}
\int_D w^2|x|^{\mu-2}\,dx\leq 
5\int_D|\nabla w|^2|x|^{\mu}\,dx.
\]
The proof is complete. 
\end{proof}
Here we consider the estimate for $E_{\mu+1}^\sharp$. 
The following two lemmas are the estimates 
for the derivatives of $\widetilde{E}_{\mu+1}$ 
and $E_{\mu}^*$, respectively. 
\begin{lemma}\label{lem:energy1}
Assume that 
$(w_0,w_1)\in H^1_0(D)\times L^2(D)$ 
and  $F\in C([0,T];L^2(D))$.
Let $w$ be the solution of \eqref{AP}. 
Then for every $t\in (0,T)$, 
\begin{align*}
\frac{d}{dt}\widetilde{E}_{\mu+1}(w;t)
&\leq 
-2V_0
\int_D(\pa_tw)^2|x|^{\mu}\,dx
+
2
\int_D\pa_twF|x|^{\mu+1}\,dx
\\
&\quad
+2(\mu+1)
\left(\int_D(\pa_tw)^2|x|^{\mu}\,dx\right)^{\frac{1}{2}}
\left(\int_D|\nabla_{\!\mu}w|^2|x|^{\mu}\,dx\right)^{\frac{1}{2}}.
\end{align*}
\end{lemma}
\begin{proof}
By the definitions of $\widetilde{E}_{\mu+1}$ and $\nabla_{\!\mu}$, 
we see from integration by parts that
\begin{align*}
\frac{d}{dt}\widetilde{E}_{\mu+1}(w;t)
&=
2\int_D\Big(\pa_tw\pa_t^2w+\nabla \pa_tw\cdot\nabla w\Big)|x|^{\mu+1}\,dx
\\
&\quad-(\mu+1)(N-2+\mu)
\int_D w\pa_tw|x|^{\mu-1}\,dx
\\
&=
2\int_D \pa_tw\Big(\pa_t^2w-\Delta w\Big)|x|^{\mu+1}\,dx
-2(\mu+1)\int_D \pa_tw\nabla w\cdot x|x|^{\mu-1}\,dx
\\
&\quad
-(\mu+1)(N-2+\mu)
\int_D w\pa_tw|x|^{\mu-1}\,dx
\\
&=
2\int_D \pa_tw\Big(\pa_t^2w-\Delta w\Big)|x|^{\mu+1}\,dx
-2(\mu+1)\int_D
\nabla_{\!\mu}w\cdot \frac{x}{|x|}\pa_tw|x|^{\mu}\,dx.
\end{align*}
Using the equation in \eqref{AP} and the Young inequality, 
we deduce the desired inequality. 
\end{proof}
\begin{lemma}\label{lem:energy2}
Assume that 
$(w_0,w_1)\in H^1_0(D)\times L^2(D)$ 
and  $F\in C([0,T];L^2(D))$.
Let $w$ be the solution of \eqref{AP}. 
Then for every $t\in (0,T)$, 
\begin{align}
\nonumber 
\frac{d}{dt}E_{\mu}^*(w;t)
&=
2
\int_D(\pa_tw)^2|x|^{\mu}\,dx
-2
\int_D |\nabla_{\!\mu}w|^2|x|^{\mu}\,dx
\\
\label{eq:E^*}
&\quad
-\frac{(N-2)^2-\mu^2}{2}
\int_D w^2|x|^{\mu-2}\,dx
+
2
\int_D w(t)F(t)|x|^{\mu}\,dx.
\end{align}
\end{lemma}
\begin{proof}
By using the equation in \eqref{AP}, we have 
\begin{align*}
\frac{d}{dt}E_{\mu}^*(w;t)
&=
2
\int_D (\pa_tw)^2|x|^{\mu}\,dx
+
2\int_D w\Big(\pa_t^2w+\frac{V_0}{|x|}\pa_tw\Big)|x|^{\mu}\,dx
\\
&=
2
\int_D (\pa_tw)^2|x|^{\mu}\,dx
+
2\int_D w\Big(\Delta w+F\Big)|x|^{\mu}\,dx
\\
&=
2
\int_D (\pa_tw)^2|x|^{\mu}\,dx
+
2\int_D w\Delta w|x|^{\mu}\,dx
+
2\int_D wF|x|^{\mu}\,dx.
\end{align*}
Observing that 
integration by parts provides
\begin{align*}
-\int_D w\Delta w|x|^{\mu}\,dx
&=
\int_D |\nabla w|^2|x|^{\mu}\,dx
+
\mu\int_D w\nabla w\cdot x|x|^{\mu-2}\,dx
\\
&=
\int_D \left|\nabla w+\frac{N-2+\mu}{2}\frac{x}{|x|^2}w\right|^2|x|^{\mu}\,dx
\\
&\quad-(N-2)
\int_D w\nabla w\cdot x|x|^{\mu-2}\,dx
-
\left(\frac{N-2+\mu}{2}\right)^2\int_D w^2|x|^{\mu-2}\,dx
\\
&=
\int_D |\nabla_{\!\mu}w|^2|x|^{\mu}\,dx
+
\frac{(N-2)^2-\mu^2}{4}
\int_D w^2|x|^{\mu-2}\,dx, 
\end{align*}
we obtain \eqref{eq:E^*}. 
\end{proof}
The following lemma is the estimate 
for the derivative of $E_{\mu+1}^{\sharp}$, 
which is a summary of Lemmas 
\ref{lem:energy1} and \ref{lem:energy2}. 
\begin{lemma}\label{lem:energy3}
Assume that 
$(w_0,w_1)\in H^1_0(D)\times L^2(D)$ 
and  $F\in C([0,T];L^2(D))$.
Let $w$ be the solution of \eqref{AP}. 
Then for every $t\in (0,T)$, 
\begin{align*}
\frac{d}{dt}
E_{\mu+1}^{\sharp}(w;t)
&\leq 
-(V_0-|\mu+1|)
\int_D \Big(|\pa_tw(t)|^2+|\nabla_{\!\mu}w(t)|^2\Big)|x|^{\mu}\,dx
\\
&\quad 
-\frac{[(N-2)^2-\mu^2]V_0}{4}
\int_D w^2|x|^{\mu-2}\,dx
\\
&
\quad
+2
\int_\Omega \pa_tw(t)F(t)|x|^{\mu+1}\,dx
+
V_0
\int_\Omega w(t)F(t)|x|^{\mu}\,dx. 
\end{align*}
In particular, if $|\mu+1|<V_0$, then
there exist positive constants $\delta_\mu'$ and $C_{\mu+1}'$ such that 
\begin{align*}
&\frac{d}{dt}E_{\mu+1}^{\sharp}(w;t)
+
\delta_\mu' E_{\mu}^{\sharp}(w;t)
\\
&\leq 
\begin{cases}
C_{\mu+1}'
\displaystyle\int_D F^2|x|^{\mu+2}\,dx
&\text{if}\ |\mu|<N-2,
\\
C_{\mu+1}'
\displaystyle
\left(
\int_D w^2|x|^{\mu-2}\,dx+
\int_D F^2|x|^{\mu+2}\,dx\right)
&\text{if}\ |\mu|\geq N-2.
\end{cases}
\end{align*}
\end{lemma}
\begin{remark}
In the homogeneous case $F\equiv 0$,  
we have from the first inequality in Lemma \ref{lem:energy3} that if $V_0\leq N-1$, then choosing $\mu=V_0-1$, we have
\[
c_{V_0}^\sharp E_{V_0}(w,t)
\leq C_{V_0}^\sharp 
E_{V_0}(w;0).
\]
Proceeding the proof of Ikehata--Todorova--Yordanov \cite[Proposition 2.2]{ITY2013}, we can deduce the energy decay estimate 
\[
E_{0}(w;t)\leq Ct^{-V_0}
E_{V_0}(w,0)
\]
which is the same as \cite[Theorem 1.1]{ITY2013} (for $1<V_0\leq N-1$). 
\end{remark}
Here we prove  Proposition \ref{prop:3-1} under the restriction $m<N-2$. 
\begin{proof}[Proof of Proposition \ref{prop:3-1} when $0<m<N-2$]
We see from Lemma \ref{lem:energy3} with $\mu=m$ that
\begin{align*}
E_{m+1}^{\sharp}(w;t)+\delta_m' 
\int_0^tE_{m}^{\sharp}(w,s)\,ds
&\leq 
E_{m+1}^{\sharp}(w;0)+
 C_{m+1}'\int_0^t\!\!\int_D F(s)^2|x|^{m+2}\,dx\,ds
\end{align*}
Moreover, Lemma \ref{lem:energy3} with $\mu=m-1>-1$ gives
\begin{align*}
&\frac{d}{dt}
\Big[
 (1+t)E_{m}^{\sharp}(w;t)
\Big]
+\delta_{m-1}' 
(1+t)E_{m-1}^{\sharp}(w;t)
\\
&\leq 
(1+t)\Big[\frac{d}{dt}E_{m}^{\sharp}(w;t)
+\delta_{m-1}' 
E_{m-1}^{\sharp}(w;t)\Big]
+E_{m}^{\sharp}(w;t)
\\
&\leq 
C_{m}'(1+t)
\int_D F^2|x|^{m+1}\,dx
+E_{m}^{\sharp}(w,t)
\end{align*}
and therefore we have 
\begin{align*}
&(1+t)E_{m}^{\sharp}(w;t)
+
\delta_{m-1}' 
\int_0^t(1+s)E_{m-1}^{\sharp}(w;s)\,ds
\\
&\leq 
E_{m}^{\sharp}(w;0)+ 
C_{m}'
\int_0^t(1+s)\int_D F(s)^2|x|^{m+1}\,dx\,ds
+\frac{C_{m+1}'}{\delta_{m}'}
\int_0^t\!\!\int_D F(s)^2|x|^{m+2}\,dx\,ds.
\end{align*}
Noting that for every $\mu\in [0,m+1]$, 
$E_{\mu}^{\sharp}(w;0)\leq C_\mu^\sharp E_{m+1}^{\Psi}(w,0), 
$ (by Lemma \ref{lem:energy0})
and 
$(1+t)^{\mu}|x|^{m-\mu+2}\leq |x|\Psi^{m+1}$, 
we can deduce by iteration that 
choosing $k\in \N\cap[m,m+1)$, we have
\begin{align*}
&(1+t)^kE_{m+1-k}^{\sharp}(w;t)
+
\widetilde{\delta} 
\int_0^t(1+s)^{k}E_{m-k}^{\sharp}(w;s)\,ds
\\
&\leq 
\widetilde{C}\left(
E_{m}^{\Psi}(w;0)+ 
\int_0^t\!\!\int_D F(s)^2|x|\Psi^{m+1}\,dx\,ds
\right)
\end{align*}
for some positive constants $\widetilde{\delta}$ and $\widetilde{C}$.
Since 
\[
(1+s)^{m}E_0(w;s)\leq 
\Big((1+s)^{k}E_{m-k}(w;s)\Big)^{\frac{m}{k}}
\Big(E_{m}(w,s)\Big)^{1-\frac{m}{k}}, 
\]
we have 
\[
\int_{0}^tE_{m}^{\Psi}(w;s)\,ds
\leq 
\widetilde{C}'\left(
E_{m}^{\Psi}(w;0)+ 
\int_0^t\!\!\int_D F(s)^2\Psi(s)^{m+1}|x|\,dx\,ds
\right)
\]
for some positive constant $\widetilde{C}'$.
Furthermore, noting that 
\begin{align*}
\frac{d}{dt}
\Big[(1+t)^{m+1}E_0(w;t)\Big]
&=
m(1+t)^{m}E_0(w;t)+2(1+t)^{m+1}\int_D\pa_tw\left(-\frac{V_0}{|x|}\pa_tw+F\right)dx
\\
&\leq 
m(1+t)^{m}E_0(w;t)+\frac{1}{2V_0}(1+t)^{m+1}\int_DF^2|x|dx
\\
&\leq 
m(1+t)^{m}E_0(w;t)+\frac{1}{2V_0}\int_DF^2\Psi^{m+1}|x|dx,
\end{align*}
we obtain the desired estimate for $E^{\Psi}_{m+1}(w;t)$. 
The proof is complete.
\end{proof}
\subsection{Energy estimates for \texorpdfstring{$N-2\leq m<N-1$}{}}
In this case, we introduce 
auxiliary functions $\psi_1$, $v$ and $U$
which are given 
as the solutions of the following elliptic, parabolic and hyperbolic problems, 
respectively: 
\begin{gather}\label{eq:psi1}
\begin{cases}
\dfrac{\lambda}{|x|^2}\psi_1-\Delta \psi_1=w_1
& \text{in}\ D,
\\
\psi_1=0
& \text{on}\ \pa D, 
\end{cases}
\\
\label{eq:v}
\begin{cases}
\dfrac{V_0}{|x|}\pa_tv-\Delta v=F
& \text{in}\ D\times(0,T),
\\
v=0
& \text{on}\ \pa D\times(0,\infty),
\\
v(0)=v_0:=w_0+\dfrac{\lambda}{V_0|x|}\psi_1, 
\end{cases}
\\
\label{eq:U}
\begin{cases}
\pa_t^2U-\Delta U+\dfrac{V_0}{|x|}\pa_tU=-\pa_tv
& \text{in}\ D\times(0,T),
\\
U=0
& \text{on}\ \pa D\times(0,T),
\\
(U,\pa_tU)(0)=\Big(-\psi_1,-\dfrac{\lambda}{V_0|x|}\psi_1\Big),   
\end{cases}
\end{gather}
where $(w_0,w_1)\in H^1_0(D)\times L^2(D)$, $F\in C([0,T);L^2(D))$
and $\lambda=\max\limits_{-1\leq \mu\leq m}\lambda_\mu$ 
with $\lambda_\mu=\frac{(\mu-1)(N-3+\mu)}{2}+1$. 

The following lemma provides 
the existence of a unique solution $\psi_1$ of \eqref{eq:psi1}
and its estimate.  
\begin{lemma}\label{lem:ell}
Assume that $w_1\in L^2(D)$.
Then there exists a unique solution $\psi_1\in H^2(D)\cap H^1_0(D)$ 
of \eqref{eq:psi1} for every $\lambda\geq 0$. 
Moreover, 
if $\mu\in\R$ and $\lambda\geq \lambda_\mu=\frac{(\mu-1)(N-3+\mu)}{2}+1$, then 
\[
\int_D \psi_1^2|x|^{\mu-3}\,dx
+
\int_D|\nabla \psi_1|^2|x|^{\mu-1}\,dx
\leq 
\int_D w_1^2|x|^{\mu+1}\,dx.
\]
In particular, if $w_0\in H^1_0(D)$, then one has
\begin{align*}
\int_D |\nabla v_0|^2|x|^{\mu+1}\,dx
\leq 
3
\int_D |\nabla w_0|^2|x|^{\mu+1}\,dx
+
\frac{3\lambda^2}{V_0^2}
\int_D w_1^2|x|^{\mu+1}\,dx. 
\end{align*}
\end{lemma}
\begin{proof}
Noting that $0$ is in the resolvent of $-\Delta$ in bounded domain
and $\frac{\lambda}{|x|^2}$ is nonnegative 
and bounded, we see 
that there exists a unique solution $\psi_1\in H^2(D)\cap H^1_0(D)$. 
Then by integration by parts twice, we see 
\begin{align*}
\int_{D}(-\Delta \psi_1)\psi_1|x|^{m-1}\,dx
&=
\int_{D}|\nabla \psi_1|^2|x|^{m-1}\,dx
+(m-1)
\int_{D}\psi_1\nabla \psi_1 \cdot x|x|^{m-3}\,dx
\\
&=
\int_{D}|\nabla \psi_1|^2|x|^{m-3}\,dx
-
\frac{(m-1)(N-3+m)}{2}
\int_{D}\psi_1^2|x|^{m-3}\,dx.
\end{align*}
Therefore if $\lambda\geq \frac{(m-1)(N-3+m)}{2}+1$, we have
\begin{align*}
\int_D\psi_1^2|x|^{m-3}\,dx
+
\int_D|\nabla \psi_1|^2|x|^{m-1}\,dx
&\leq 
\int_D \left(\frac{\lambda}{|x|^2}\psi_1-\Delta\psi_1\right)\psi_1|x|^{m-1}\,dx
\\
&\leq 
\left(\int_D w_1^2|x|^{m+1}\,dx\right)^{\frac{1}{2}}
\left(\int_D \psi_1^2|x|^{m-3}\,dx\right)^{\frac{1}{2}}. 
\end{align*}
Therefore we obtain the desired inequality. 
\end{proof}

Next we consider the existence of a unique solution 
$v$ of \eqref{eq:v} and space-time weighed estimates of $v$ and $\pa_tv$. 

\begin{lemma}\label{lem:est-v1}
Assume that $(w_0,w_1)\in H^1_0(D)\times L^2(D)$ and $F\in C([0,T);L^2(D))$.
Then there exists a unique solution $v$ of \eqref{eq:v}. 
Moreover, $v$ satisfies
\begin{align*}
&\int_D \frac{v(t)^2}{|x|}\Psi(t)^{m}\,dx
+
\int_0^t
	\!\!\int_D |\nabla v(s)|^2\Psi(s)^{m}\,dx
\,ds
\\
&\leq 
C_{m,1}\Big(
\int_D \Big(|\nabla w_0|^2+w_1^2\Big)(1+|x|)^{m+1}\,dx
+
\int_0^t\!\!\int_\Omega F(s)^2\Psi(s)^{m+1}|x|\,dx\,ds\Big)
\end{align*}
for some positive constant $C_{m,1}$ depending only on $N$, $m$ and $V_0$.
\end{lemma}
\begin{proof}
Choose $\beta$ satisfying $m<\beta<N-1$ and $\beta(1-2\delta)=m$. 
Then using 
the equation in \eqref{eq:v} and 
Lemmas \ref{lem:Phi} {\bf (i)} and \ref{lem:ibp}, 
we have
\begin{align*}
\frac{d}{dt}
\int_D \frac{V_0}{|x|}\frac{v^2}{\Phi_{\beta}^{1-2\delta}}\,dx
&=
2\int_D \frac{V_0}{|x|}\frac{v\pa_tv}{\Phi_{\beta}^{1-2\delta}}\,dx
-
(1-2\delta)\int_D \frac{V_0}{|x|}\frac{v^2}{\Phi_{\beta}^{2-2\delta}}\pa_t\Phi_\beta\,dx
\\
&=
2\int_D \frac{v(\Delta v+G)}{\Phi_{\beta}^{1-2\delta}}\,dx
-
(1-2\delta)\int_D \frac{v^2}{\Phi_{\beta}^{2-2\delta}}\Delta\Phi_\beta\,dx
\\
&\leq 
-
\frac{2\delta}{1-\delta}
\int_D \frac{|\nabla v|^2}{\Phi_{\beta}^{1-2\delta}}\,dx
+
2
\left(
\int_D \frac{v^2}{|x|\Psi\Phi_{\beta}^{1-2\delta}}\,dx
\right)^{\frac{1}{2}}
\left(
\int_D \frac{G^2|x|\Psi }{\Phi_{\beta}^{1-2\delta}}\,dx
\right)^{\frac{1}{2}}
\\
&\leq 
-
\frac{2\delta}{1-\delta}
\int_D \frac{|\nabla v|^2}{\Phi_{\beta}^{1-2\delta}}\,dx
+
\ep 
\int_D \frac{v^2}{|x|\Psi\Phi_{\beta}^{1-2\delta}}\,dx
+
\frac{1}{\ep}
\int_D \frac{G^2|x|\Psi }{\Phi_{\beta}^{1-2\delta}}\,dx.
\end{align*}
Using the above estimate and integrating it over $[0,t]$, we deduce
\begin{align*}
&\int_D \frac{V_0}{|x|}\frac{v(t)^2}{\Phi_{\beta}(t)^{1-2\delta}}\,dx
+
\frac{2\delta}{1-\delta}
\int_0^t\!\!\int_D \frac{|\nabla v(s)|^2}{\Phi_{\beta}(s)^{1-2\delta}}\,dx\,ds
\\
&\leq 
\int_D \frac{V_0}{|x|}\frac{v_0^2}{\Phi_{\beta}(0)^{1-2\delta}}\,dx
+
\ep 
\int_0^t\!\!\int_D \frac{v(s)^2}{|x|\Psi(s)\Phi_{\beta}(s)^{1-2\delta}}\,dx\,ds
+
\frac{1}{\ep}
\int_0^t\!\!\int_D \frac{F^2(s)\Psi(s)|x|}{\Phi_{\beta}(s)^{1-2\delta}}\,dx\,ds
\end{align*}
and therefore,  
it follows from Lemma \ref{lem:Phi} {\bf (iii)}, {\bf (iv)} 
that
\begin{align*}
&V_0\int_D v(t)^2\frac{\Psi(t)^m}{|x|}\,dx
+
\frac{2\delta}{1-\delta}
\int_0^t\!\!\int_D |\nabla v(s)|^2\Psi(s)^m\,dx\,ds
\\
&\leq 
C\left(
\int_D v_0^2\frac{\Psi(0)^m}{|x|}\,dx
+
\ep 
\int_0^t\!\!\int_D v(s)^2\frac{\Psi(s)^{m-1}}{|x|}\,dx\,ds
+
\frac{1}{\ep}
\int_0^t\!\!\int_D F^2(s)\Psi(s)^{m+1}|x|\,dx\,ds
\right),
\end{align*}
where $C=(C_{\Phi_\beta}'/c_{\Phi_\beta}')^{1-2\delta}$
with 
$C_{\Phi_\beta}'=C_{\Phi_\beta}\max\{1,V_0\}$
and 
$c_{\Phi_\beta}'=c_{\Phi_\beta}\min\{1,V_0\}$.
Noting that by Lemma \ref{lem:hardy2}
\begin{align*}
\int_D v_0^2\frac{\Psi(0)^m}{|x|}\,dx
&\leq 
\min\left\{\frac{N-1}{2},\frac{N-1+m}{2}\right\}^{-2}\int_D |\nabla v_0|^2\Psi(0)^{m+1}\,dx
\\
\int_0^t\!\!\int_D v(s)^2\frac{\Psi(s)^m}{|x|}\,dx\,ds
&\leq 
\min\left\{\frac{N-1}{2},\frac{N-2+m}{2}\right\}^{-2}
\int_0^t\!\!\int_D |\nabla v(s)|^2\Psi(s)^{m}\,dx\,ds, 
\end{align*}
applying Lemma \ref{lem:ell} with $\mu=-1,m$, 
and choosing 
$\ep$ small enough, we obtain the desired estimate.
\end{proof}

Next lemma gives a weighted estimate for $\pa_tv$ and $\nabla v$. 

\begin{lemma}\label{lem:est-v2}
Assume that $(w_0,w_1)\in H^1_0(D)\times L^2(D)$ and $F\in C([0,T);L^2(D))$.
Let $v$ be the solution of \eqref{eq:v}. 
Then 
\begin{align*}
&\int_D |\nabla v(t)|^2\Psi(t)^{m+1}\,dx
+
\int_0^t\int_D \frac{|\pa_tv(s)|^2}{|x|}\Psi(s)^{m+1}\,dx\,dt
\\
&\leq 
C_{m,2}'
\left(
\int_D \Big(|\nabla w_0|^2+w_1^2\Big)(1+|x|)^{m+1}\,dx
+
\int_0^t\int_D F(s)^2\Psi(s)^{m+1}|x|\,dx\,ds\right)
\end{align*}
for some positive constant $C_{m,2}'$ depending only on $N$, $m$ and $V_0$.
\end{lemma}
\begin{proof}
Using the equation in \eqref{eq:v}, we have
\begin{align*}
\frac{d}{dt}
\int_D |\nabla v|^2\Psi^{m+1}\,dx
&=
2\int_D \nabla \pa_tv\cdot\nabla v\Psi^{m+1}\,dx
+
(m+1)\int_D |\nabla v|^2\Psi^{m}\,dx
\\
&=
-
2\int_D \pa_tv\Delta v\Psi^{m+1}\,dx
-2(m+1)\int_D \pa_tv\nabla v\cdot\frac{x}{|x|}\Psi^{m}\,dx
\\
&\quad +
(m+1)\int_D |\nabla v|^2\Psi^{m}\,dx
\\
&\leq 
-2V_0\int_D \frac{(\pa_tv)^2}{|x|}\Psi^{m+1}\,dx
+
(m+1)\int_D |\nabla v|^2\Psi^{m}\,dx
\\
&\quad
+2(m+1)\int_D |\pa_tv||\nabla v|\Psi^{m}\,dx
+2\int_D \pa_tv F\Psi^{m+1}\,dx.
\end{align*}
By the Young inequality, we have
\begin{align*}
(m+1)\int_D |\pa_tv||\nabla v|\Psi^{m}\,dx
&\leq 
\frac{V_0}{2}\int_D (\pa_tv)^2\Psi^{m}\,dx
+
\frac{2(m+1)^2}{V_0}\int_D |\nabla v|^2\Psi^{m}\,dx
\\
&\leq 
\frac{V_0}{2}\int_D (\pa_tv)^2\frac{\Psi^{m+1}}{|x|}\,dx
+
\frac{2(m+1)^2}{V_0}\int_D |\nabla v|^2\Psi^{m}\,dx
\end{align*}
and 
\begin{align*}
2\int_D \pa_tv F\Psi^{m+1}\,dx\leq 
\frac{V_0}{2}\int_D (\pa_tv)^2\frac{\Psi^{m+1}}{|x|}\,dx
+\frac{2}{V_0}\int_D F^2\Psi^{m+1}|x|\,dx. 
\end{align*}
Combining the above three estimates, 
we obtain the desired estimate.
\end{proof}

The following is the reason why we introduce 
the problems \eqref{eq:psi1}--\eqref{eq:U}. 
The solution $w$ of \eqref{AP} can be represented by 
the sum of two parts which are the solutions of \eqref{eq:v} and \eqref{eq:U}. 
The idea of this kind of decomposition for the abstract evolution equations is introduced in Sobajima \cite{Soba_AE} and Ikehata--Sobajima \cite{IkeSob_SLP}. 
\begin{lemma}\label{lem:decomp}
Assume that $(w_0,w_1)\in H^1_0(D)\times L^2(D)$ and $F\in C([0,T);L^2(D))$.
Let $w$, $v$ and $U$ be the solutions of \eqref{AP}, \eqref{eq:v} and \eqref{eq:U}, respectively. Then $w=v+\pa_tU$. 
\end{lemma}
\begin{proof}
Observe that ${}^t(-\psi_1,-\frac{\lambda}{V_0|x|}\psi_1)
\in (H^2(D)\cap H^1_0(D)) \times H^1_0(\Omega)$. 
Let $U_n$ be a solution of 
\begin{equation}\label{eq:U_n}
\begin{cases}
\pa_t^2U_n-\Delta U_n+\dfrac{V_0}{|x|}\pa_tU_n=-\pa_tv
\quad \text{in}\ D\times(0,T),
\\
U_n(x,t)=0
\quad \text{on}\ \pa D\times(0,T),
\\
(U_n,\pa_tU_n)(0)=(-J_n\psi_1,-J_n(\frac{\lambda}{V_0|x|}\psi_1)),
\end{cases}
\end{equation}
where $J_n=(1-n^{-1}\Delta)^{-1}$. 
Then we have $\Delta U_n\in C^1([0,T);L^2(D))$
and 
\[
{}^t(U_n(t),\pa_tU_n(t))\to {}^t(U(t),\pa_tU(t))  \text{in}\ (H^2(D)\cap H^1_0(D))\times H^1_0(D)
\]
as $n\to \infty$. 

Put $w_n=v+\pa_tU_n$, and therefore, 
we have $w_n(t)\to v(t)+\pa_tU(t)$ 
in $H^1_0(D)$ as $n\to \infty$. 
On the other hand,  we see from \eqref{eq:U_n} that 
\[
\pa_tw_n=\pa_tv+\pa_t^2U_n=
\Delta U_n-\frac{V_0}{|x|}\pa_tU_n
=\Delta U_n-\frac{V_0}{|x|}(w_n-v).
\]
This gives $w_n(0)=w_0+U_1-J_nU_1\to w_0$  
in $H^1_0(D)$
and 
$\pa_tw(0)
=-J_n\Delta \psi_1+\frac{\lambda}{|x|}J_n(\frac{\psi_1}{|x|})
\to
-\Delta \psi_1+\frac{\lambda}{|x|^2}\psi_1
= w_1$ in $L^2(D)$ as $n\to \infty$. 
Moreover, we have
\[
\pa_t^2w_n
+
\frac{V_0}{|x|}
\pa_tw_n=
\Delta\pa_tU_n+\frac{V_0}{|x|}\pa_tv
=\Delta\pa_tU_n+\Delta v+F
=\Delta w_n+F. 
\]
The uniqueness of solutions for the homogeneous problem of \eqref{eq:U_n}
implies 
$w_n(t)\to w(t)$ in $H^1_0(D)$ as $n\to \infty$. 
The proof 
is complete. 
\end{proof}
Using the property of $U$, now we prove the weighted $L^2$-estimates for the solution $w$ of \eqref{AP}.  
\begin{proposition}\label{prop:est-w}
Assume that $(w_0,w_1)\in H^1_0(D)\times L^2(D)$ and $F\in C([0,T);L^2(D))$.
Let $w$ be the solution of \eqref{AP}. Then
\begin{align*}
&
\int_D w(t)^2|x|^{m-1}\,dx
+
\int_0^t 
\int_D w(s)^2|x|^{m-2}\,dx
\,ds
\\
&\leq 
C_{m,3}'
\left(
\int_D \Big(|\nabla w_0|^2+w_1^2\Big)(1+|x|)^{m+1}\,dx
+
\int_0^t
\int_D F(s)^2\Psi^{m+1}(s)|x|\,dx\,ds
\right).
\end{align*}
for some positive constant $C_{m,3}'$ depending only on $N$, $m$ and $V_0$.
\end{proposition}
\begin{proof}
First we note that by 
Lemmas \eqref{lem:hardy}, \eqref{lem:est-v1} and \ref{lem:est-v2}, we already have
\begin{align}
\nonumber
&\int_D v(t)^{2}|x|^{m-1}\,dx
+
\int_0^t\!\!\int_D v(s)^{2}|x|^{m-2}\,dx\,ds
\\
\nonumber
&\leq \left(\frac{N-1+m}{2}\right)^2\int_D |\nabla v(t)|^{2}|x|^{m+1}\,dx
+
\left(\frac{N-2+m}{2}\right)^2\int_0^t\!\!\int_D |\nabla v)(s)|^{2}|x|^{m}\,dx\,ds
\\
\nonumber
&\leq \left(\frac{N-2+m}{2}\right)^2\int_D |\nabla v(t)|^{2}\Psi(t)^{m+1}\,dx
+
\left(\frac{N-1+m}{2}\right)^2\int_0^t\!\!\int_D |\nabla v(s)|^{2}\Psi(s)^{m}\,dx\,ds
\\
\label{est:v}
&\leq C\left(
\int_D \Big(w_0^2+w_1^2\Big)(1+|x|)^{m+1}\,dx
+
\int_0^t\int_D F(s)^2\Psi(s)^{m+1}|x|\,dx\,ds\right)
\end{align}
for some positive constant $C$. 
Applying Lemma \ref{lem:energy3} with $\mu=m-2\in (-1,N-2)$ 
to the solution $U$ of \eqref{eq:U}, we have 
\begin{align*}
&E_{m-1}^\sharp(U;t)
+
\delta_{\mu-2}\int_0^t E_{m-2}^\sharp(U;s)\,ds
\\
&\leq 
C_{\mu-1}'
\left(
\int_D\Big(|\nabla U(0)|^2+\big(\pa_tU(0)\big)^2\Big)|x|^{m-1}\,dx
+
\int_0^t
\int_D |\pa_tv(s)|^2|x|^{m}\,dx\,ds
\right)
\\
&\leq 
C_{\mu-1}'
\left(
\int_D |\nabla \psi_1|^2|x|^{m-1}\,dx
+
\frac{\lambda^2}{V_0^2}
\int_D \psi_1^2|x|^{m-3}\,dx
+
\int_0^t
\int_D |\pa_tv(s)|^2\frac{\Psi(s)^{m+1}}{|x|}\,dx\,ds
\right).
\end{align*}
Applying Lemmas \ref{lem:ell}, \ref{lem:est-v2} and \ref{lem:energy0} 
provides
\begin{align*}
&E_{m-1}(U;t)
+
\delta_{\mu-2}\int_0^t E_{m-2}(U;s)\,ds
\\
&\leq 
C'
\left(
\int_D \Big(|\nabla w_0|^2+w_1^2\Big)|x|(1+|x|)^m\,dx
+
\int_0^t\int_D F(s)^2\Psi(s)^{m+1}|x|\,dx\,ds
\right)
\end{align*}
for some positive constant $C'$. 
Noting that $w=v+\pa_tU$ and 
\[
E_{m-1}(U;t)\geq \int (\pa_tU)^2|x|^{m-1}\,dx,
\] 
we obtain the desired inequality. 
\end{proof}

At the end of this subsection, 
we prove Proposition \ref{prop:3-1} in the rest case. 

\begin{proof}[Proof of Proposition \ref{prop:3-1} when $m\in [N-2,N-1)$]
In view of Lemma \ref{lem:energy3} for $\mu=m\geq N-2$, we already have
\begin{align*}
\frac{d}{dt}E_{m+1}^{\sharp}(w;t)
+
\delta_m' E_{m}^{\sharp}(w;t)
&\leq 
C_{m+1}'
\left(
\int_D w^2|x|^{m-2}\,dx+
\int_D F^2|x|^{m+2}\,dx\right)
\\
&\leq 
C_{m+1}'
\left(
\int_D w^2|x|^{m-2}\,dx+
\int_D F^2\Psi^{m+1}|x|\,dx\right).
\end{align*}
By virtue of Proposition \ref{prop:est-w}, 
integrating the above inequality on $[0,t]$
and using Lemma \ref{lem:energy0}, we deduce
\[
E_{m+1}^\sharp(w;t)
+
\delta_m'\int_0^tE_{m}^\sharp(w;s)\,ds
\leq 
C
\left(
E_{m+1}^\Psi(w;0)
+\int_0^t\!\!\int_D F(s)^2\Psi(s)^{m+1}|x|\,dx\,ds\right)
\]
for some positive constant $C$. 
Then noting that $m+1<V_0$ yields 
the validity of  assumption of Lemma \ref{lem:energy0} 
and proceeding the same argument as in the case of $0<m<N-2$, 
we can obtain the desired estimate for $N-2\leq m<N-1$. 
The proof is complete.
\end{proof}

\section{The semilinear problem in exterior domains}\label{sec:GE}

To prove global existence for the semilinear problem \eqref{P}, 
we use the following lemma which is so-called blowup alternative. 
\begin{lemma}\label{lem:alternative}
Assume that ${}^t(u_0,u_1)\in H^1_0(\Omega)\times L^2(\Omega)$. 
Let $u$ be the weak solution of \eqref{P} in $(0,T_*)$ 
with the corresponding lifespan $T_{\max}$, that is, 
\[
T_{\max}=\sup\{T\in (0,\infty]\;;\;\text{there exists a weak solution of \eqref{P} in $(0,T)$}\}.
\]
If $T_{\max}<\infty$, then 
$\lim\limits_{t\nearrow T_{\max}}\big(\|u(t)\|_{H^1_0(\Omega)}+\|\pa_tu(t)\|_{L^2(\Omega)}\big)=\infty$.  
\end{lemma}

\begin{proof}[Proof of Theorem \ref{thm:main}]
The proof is divided into two steps
which are 
for the case of compactly supported initial data
and 
the one of initial data with non-compact supports. 

{\bf Step 1 (compactly supported initial data).}  
Let $u$ be a weak solution of \eqref{P} in $(0,T)$ 
and set ${\rm supp}\,u_0\cup {\rm supp}\,u_1\subset \overline{B(0,R_0)}$
for some $R_0>R_\Omega$. 
Choose $D=\Omega\cap B(0,R_0+T)$. 
Since the problem \eqref{P} has the finite propagation property,   
$u_D=u|_D$ (the restriction on $D$) can be regarded as the solution of 
\begin{align}\label{PonD}
\begin{cases}
\pa_t^2u_D(x,t)-\Delta u_D(x,t) +\dfrac{V_0}{|x|}\pa_tu_D(x,t)
=|u_D(x,t)|^{p-1}u_D(x,t)
\quad \text{in}\ D\times(0,T)
\\
u_D(x,t)=0\quad \text{on}\ \pa D\times(0,T), 
\\
(u_0,u_1)(0)=(u_0|_D,u_1|_D).
\end{cases}
\end{align}
Here we take $m=\frac{4}{p-1}-N+1\in (N-3,N-1)$, that is, 
$m$ satisfies $p=1+\frac{4}{N-1+m}$.
Then the assumption for the initial value can be written as follows:
\[
E_{m+1}^{\Psi}(u_D;0)
=\int_\Omega \Big(|\nabla u_0|^2+u_1^2\Big)(1+|x|)^{m+1}\,dx\leq \delta,
\]
where $\delta$ will be chosen later. 
By Proposition \ref{prop:3-1}, we have 
\begin{align}
\nonumber
&E_{m+1}^{\Psi}(u_D;t)
+
\delta_m
\int_0^tE_{m}^{\Psi}(u_D;s)
\,ds
\\
\label{est:NDW}
&\leq 
K_m
\left(
E_{m+1}^{\Psi}(u_D;0)
+
\int_0^t\int_D |u_D(s)|^{2p}|x|\Psi(s)^{m+1}\,dx\,ds
\right)
\end{align}
for $t\in [0,T)$. 
Observe that 
Lemma \ref{lem:embed-1} 
with $\mu=m+\frac{p-1}{p}$ 
and 
$\mu=0$ 
respectively implies 
\begin{align*}
\int_D
	|u_D|^{2p}|x|^{m+2}
\,dx
&\leq 
C_{GN, m+\frac{p-1}{p}}\left(
\int_D
	|\nabla u_D|^2|x|^{m+\frac{p-1}{p}}
\,dx
\right)^{p}
\\
&\leq 
C_{GN, m+\frac{p-1}{p}}\left(
\int_D
	|\nabla u_D|^2|x|^{m+1}
\,dx
\right)^{p-1}
\int_D
	|\nabla u_D|^2|x|^{m}
\,dx
\\
&\leq 
C_{GN, m+\frac{p-1}{p}}\Big(
E^{\Psi}_{m+1}(u_D;t)
\Big)^{p-1}
\int_D
	|\nabla u_D|^2\Psi^{m}
\,dx
\end{align*}
and 
\begin{align*}
&(1+t)^{m+1}\int_D
	|u_D|^{2p}|x|
\,dx
\\
&\leq 
C_{GN,0}^\theta
(1+t)^{m+1}\left(\int_D
	|u_D|^{2p}|x|^{m+2}
\,dx
\right)^{1-\theta}
\left(\int_D
	|u_D|^{2p}|x|^{-2+(N-2)(p-1)}
\,dx
\right)^{\theta}
\\
&\leq 
C_{GN,0}^\theta
(1+t)^{(m+1)p\theta-\theta}\left(\int_D
	|u_D|^{2p}|x|^{m+2}
\,dx
\right)^{1-\theta}
\left(\int_D
	|\nabla u_D|^{2}
\,dx
\right)^{p\theta}
\\
&\leq 
C_{GN,0}^\theta
\left(\int_D
	|u_D|^{2p}|x|^{m+2}
\,dx
\right)^{1-\theta}
\left(
\int_D
	|\nabla u_D|^{2}\Psi^{m+1}
\,dx
\right)^{(p-1)\theta}
\left(
\int_D
	|\nabla u_D|^{2}\Psi^m
\,dx\right)^{\theta}, 
\end{align*}
where $\theta=\frac{m+1}{m+4-(N-2)(p-1)}$. 
In view of the above inequalities, 
by $(1+t)^{m+1}+|x|^{m+1}\leq 2\Psi^{m+1}$
we deduce 
\[
\int_0^t\!\!
\int_D \big||u_D(s)|^{p-1}u_D(s)\big|^2\Psi(s)^{m+1}|x|\,dx\,ds
\leq 
C
\Big(
E^{\Psi}_{m+1}(u_D;t)
\Big)^{p-1}
\int_0^tE^{\Psi}_{m}(u_D;s)\,ds
\]
for some positive constant $C$. 
Therefore putting 
\[
M_{m+1}(t)=E_{m+1}^{\Psi}(u_D;t)
+
\delta_m
\int_0^tE_{m}^{\Psi}(u_D;s)
\,ds, 
\]
by \eqref{est:NDW} we obtain 
\[
M_{m+1}(t)
\leq 
C'\Big(\delta+ \big(
M_{m+1}(t)
\big)^{p}
\Big)
\]
for some positive constant $C'$. If we choose $\delta$ sufficiently small, then
by the continuity of $M_m$
we obtain that there exists a positive constant $C''$ 
independent of $T$ such that  
\[
\int_\Omega 
\Big(|\nabla u(t)|^2+\big(\pa_t u(t)\big)^2\Big)(1+t+|x|)^{m+1}\,dx
\leq 
M_{m+1}(t)
\leq C''\delta, \quad t\in [0,T).
\]
The blowup alternative (Lemma \ref{lem:alternative})
with the above estimate implies 
$T_{\max}=\infty$, that is,  
there exists a global-in-time (weak) solution $u$
which satisfies
\[
\int_\Omega\Big(|\nabla u|^2+(\pa_tu)^2\Big)(1+t+|x|)^{m+1}\,dx
\leq C''\delta, \quad t\in [0,\infty).
\]
{\bf Step 2 (initial data with non-compact supports).}
In this case we use a family of cut-off functions $\{\zeta_n\}_{n\in\N}$ defined by
\[
\zeta_n(x)=\eta\left(\frac{\log|x|}{n}\right), \quad x\in \Omega
\]
where $\eta\in C^1(\R)$ satisfies 
$\eta(s)=1$ on $(-\infty,1]$, 
$\eta(s)=0$ on $[2, \infty)$, 
$\eta'(s)\leq 0$ on $\R$.  
Assume that 
\[
\int_\Omega \Big(|\nabla u_0|^2+u_1^2\Big)(1+|x|)^{m+1}\,dx\leq \frac{\delta}{3}.
\]
By Lemma \ref{lem:hardy}, 
we see by $1+|x|\leq 2^{\frac{m}{m+1}}(1+|x|^{m+1})^{\frac{1}{m+1}}$ that 
\begin{align*}
&
\int_\Omega \Big(|\nabla (\zeta_nu_0)|^2+(\zeta_nu_1)^2\Big)(1+|x|)^{m+1}\,dx
\\
&\leq 
\int_\Omega \zeta_n^2\Big(|\nabla u_0|^2+u_1^2\Big)(1+|x|)^{m+1}\,dx
\\
&\quad+
\frac{2\|\eta'\|_{L^\infty}}{n}\int_\Omega \zeta_n|u_0\nabla u_0|
\frac{(1+|x|)^{m+1}}{|x|}\,dx
+
\frac{\|\eta'\|_{L^\infty}^2}{n^2}\int_\Omega u_0^2\frac{(1+|x|)^{m+1}}{|x|^2}\,dx
\\
&\leq 
2\int_\Omega \Big(|\nabla u_0|^2+u_1^2\Big)(1+|x|)^{m+1}\,dx
+
2^{m+1}\frac{\|\eta'\|_{L^\infty}^2}{n^2}\int_\Omega u_0^2
\left(\frac{1}{|x|^2}+|x|^{m-1}\right)\,dx
\\
&\leq 
2\int_\Omega \Big(|\nabla u_0|^2+u_1^2\Big)(1+|x|)^{m+1}\,dx
+
\frac{C'''}{n^2}\int_\Omega |\nabla u_0|^2(1+|x|^{m+1})\,dx
\end{align*}
for some positive constant $C'''$. 
Therefore if $n\geq \sqrt{C'''}$, then 
\[
\int_\Omega \Big(|\nabla (\zeta_nu_0)|^2+(\zeta_nu_1)^2\Big)(1+|x|)^{m+1}\,dx
\leq \delta. 
\]
Then by the conclusion of Step 1,  
for every $n\in \N$, the problem 
\begin{align}\label{P_n}
\begin{cases}
\pa_t^2u_n(x,t)-\Delta u_n(x,t) +\dfrac{V_0}{|x|}\pa_tu_n(x,t)
=|u_n(x,t)|^{p-1}u_n(x,t)
\quad \text{in}\ \Omega\times(0,\infty), 
\\
u_n(x,t)=0\quad \text{on}\ \pa \Omega\times(0,\infty), 
\\
(u_n,u_n)(0)=(\zeta_n u_0,\zeta_n u_1)
\end{cases}
\end{align}
has a unique (global-in-time) solution $u_n$ satisfying 
\[
\int_\Omega\Big(|\nabla u_n|^2+(\pa_tu_n)^2\Big)(1+t+|x|)^{m+1}\,dx
\leq C''\delta, \quad t\in [0,\infty).
\]
It is worth noticing that the constant $C$ is independent of $n$ (the size of support of initial data). 
Since the solution of \eqref{P} has a continuous dependence for 
the initial data in $H^1_0(\Omega)\times L^2(\Omega)$, 
letting $n\to \infty$, we can obtain that the 
sequence $\{u_n\}_{n\in\N}$ 
converges to $u$ in $C([0,\infty);H^1_0(\Omega))\cap C^1([0,\infty);L^2(\Omega))$ (uniformly in any compact interval) 
which is the global-in-time weak solution of \eqref{P}.   
The proof is complete. 
\end{proof}

\section{Remark on small data blowup when \texorpdfstring{$1<p\leq1+\frac{2}{N-1}$}{}}\label{sec:SDBU}

To end the present paper, we prove 
the blowup phenomena for arbitrary small initial data
(Proposition \ref{prop:blowup}).
\begin{proof}[Proof of Proposition \ref{prop:blowup}]
We may assume $T_{\max}>T_0=\max\{1, 2R_0,R_0^2\}$. 
Observe that $\psi(x)=1-|x|^{2-N}$ is the harmonic function 
satisfying the Dirichlet boundary condition on $\pa\Omega$. 
Fix $\eta\in C^\infty(\R)$
satisfying $\eta'(s)\leq 0$, 
$\eta(s)=1$ for $s\in (-\infty,\frac{1}{2}]$ 
and $\eta(s)=0$ for $s\in [1,\infty)$.  
Setting $\eta_T(t)=\eta(t/T)$, 
we have
\begin{align*}
\int_\Omega |u|^p\psi\eta_T^{2p'}\,dx
&=
\int_\Omega \Big(\pa_t^2u-\Delta u+\frac{V_0}{|x|}\pa_tu\Big)\psi\eta_T^{2p'}\,dx
\\
&=\frac{d}{dt}
\int_\Omega 
\Big(\pa_t u\eta_T^{2p'}
-u\pa_t(\eta_T^{2p'})+\frac{V_0}{|x|}u\eta_T^{2p'}\Big)\psi\,dx
\\
&\quad+
\int_\Omega 
u\Big(\pa_t^2\eta_T^{2p'}-\frac{V_0}{|x|}\pa_t(\eta_T^{2p'})\Big)\psi\,dx, 
\end{align*}
where $p'=p/(p-1)$ is the H\"older conjugate of $p$. 
Noting the finite propagation property ${\rm supp}\,u(t)\subset \overline{B(0,R_0+t)}$
and 
integrating the above inequality over $[0,T]$, 
we see by the Young inequality that
\begin{align*}
&\int_\Omega 
\Big(u_1+\frac{V_0}{|x|}u_0\Big)\psi\,dx
+
\int_0^T\!\!
\int_\Omega|u|^p\psi\eta_T^{2p'}\,dx\,dt
\\
&=
2p'\int_{T/2}^T\!\int_{\Omega}
u\Big(\eta_T|\pa_t^2\eta_T|+(2p'-1)|\pa_t\eta_T|^2
-\frac{V_0}{|x|}\eta_T\pa_t\eta_T\Big)\eta_T^{2p'-2}\psi\,dx\,dt
\\
&\leq 
C\|\eta\|_{W^{2,\infty}}^2
\left(
\int_{T/2}^T\!
\int_\Omega|u|^p\psi\eta_T^{2p'}\,dx\,dt
\right)^{\frac{1}{p}}
\left(
\int_{T/2}^T\!\int_{B(0,R_0+t)\setminus B(0,1)}
\Big(
\frac{1}{T^2}+\frac{1}{T|x|}\Big)^{p'}\,dx\,dt
\right)^{1-\frac{1}{p}}. 
\end{align*}
Noting that
\[
\int_{T/2}^T\!\int_{B(0,R_0+t)\setminus B(0,1)}
\Big(
\frac{1}{R^2}+\frac{1}{R|x|}\Big)^{p'}\,dx\,dt
\leq 
\begin{cases}
CT^{-\frac{1}{p-1}}
&\text{if}\ 1<p<\frac{N}{N-1},
\\
CT^{N-1-\frac{2}{p-1}}\log T
&\text{if}\ p=\frac{N}{N-1},
\\
CT^{N-1-\frac{2}{p-1}}
&\text{if}\ p>\frac{N}{N-1},
\end{cases}
\]
we can see that if $1<p<1+\frac{2}{N-1}$, then 
\[
0<\int_\Omega 
\Big(u_1+\frac{V_0}{|x|}u_0\Big)\psi\,dx
\leq 
\begin{cases}
CT^{-\frac{1}{p-1}}
&\text{if}\ 1<p<\frac{N}{N-1},
\\
CT^{N-1-\frac{2}{p-1}}\log T
&\text{if}\ p=\frac{N}{N-1},
\\
CT^{N-1-\frac{2}{p-1}}
&\text{if}\ p>\frac{N}{N-1}.
\end{cases}
\]
Since the right-hand side of the above inequality 
converge to $0$ as $T\to \infty$, 
and 
the choice of $T$ is arbitrary in $(T_0,T_{\max})$, 
the lifespan $T_{\max}$ must be finite. 
In the critical case $p=1+\frac{2}{N-1}$, taking an auxiliary function
\[
Y(T)=
\int_\Omega 
\Big(u_1+\frac{V_0}{|x|}u_0\Big)\psi\,dx+
\int_0^T
\left(\int_{\tau/2}^{\tau}\!\int_\Omega|u|^p\psi\eta_\tau^{2p'}\,dx\,dt\right)\frac{d\tau }{\tau}, 
\]
as in \cite[Lemma 3.10]{IS_NA}
we can deduce $Y(T)^p\leq CTY'(T)$ for $T\in (T_0,T_{\max})$. 
This implies that $Y(T)^{1-p}$ becomes negative if $T$ can be arbitrary large, 
which contradicts to $Y(T)>0$. 
 Therefore, we obtain  an upper bound for $T_{\max}$. The proof is complete.
\end{proof}

\subsection*{Acknowedgements}
This work is partially supported 
by JSPS KAKENHI Grant-in-Aid for Young Scientists Grant Number JP18K13445.

\newpage

\end{document}